\documentclass[twoside]{article}  
\usepackage{lipsum}
\usepackage{graphicx}
\usepackage{graphicx}
\usepackage{amssymb}
\usepackage{amsthm}
\usepackage{titlesec}
\titleformat*{\section}{\large\bfseries\centering}
\newtheorem{thm}{Theorem}
\newtheorem{cor}{Corollary}
\newtheorem{lem}{Lemma}
\newtheorem{pro}{Proposition}
\newtheorem{cla}{Claim}

\theoremstyle{remark}

\date{}
\newcommand\blfootnote[1]{
  \begingroup
  \renewcommand\thefootnote{}\footnote{#1}
  \addtocounter{footnote}{-1}
  \endgroup}

\begin{document}

\title{\textbf{Approximation of convex bodies by polytopes with respect to minimal width and diameter}}

\author{MAREK LASSAK \\
\\
Institute of Mathematics and Physics,  University of Science and \\ Technology, Kaliskiego 7, Bydgoszcz 85-789, Poland \\
e-mail: marek.lassak@utp.edu.pl}

\maketitle

\blfootnote{\textit{Key words and phrases:} Approximation, convex body, polytope, minimal width, diameter, diametral chord, reduced body}\blfootnote{\textit{Mathematics Subject Classification:} Primary 52A27.} 

\begin{abstract} Denote by ${\mathcal K}^d$ the family of convex bodies in $E^d$ and by $w(C)$ the minimal width of $C \in {\mathcal K}^d$. We ask for the greatest number $\Lambda_n ({\mathcal K}^d)$ such that every $C \in {\mathcal K}^d$ contains a polytope $P$ with at most $n$ vertices for which $\Lambda_n ({\mathcal K}^d) \leq \frac{w(P)}{w(C)}$. We give a lower estimate of $\Lambda_n ({\mathcal K}^d)$ for $n \geq 2d$ based on estimates of the smallest radius of $\big\lfloor {\frac{n}{2}} \big\rfloor$ antipodal pairs of spherical caps that cover the unit sphere of $E^d$. We show that $\Lambda_3 ({\mathcal K}^2) \geq {\frac 1 2}(3- \sqrt 3)$, and $\Lambda_n ({\mathcal K}^2) \geq \cos {\frac \pi {2 \lfloor {n/2} \rfloor}}$ for every $n \geq 4$. We also consider the dual question of estimating the smallest number $\Delta_n ({\mathcal K}^d)$ such that every $C \in {\mathcal K}^d$ there exists a polytope $P \supset C$ with at most $n$ facets for which $\frac{{\rm diam}(P)}{{\rm diam}(C)} \leq \Delta_n ({\mathcal K}^d)$. We give an upper bound of $\Delta_n ({\mathcal K}^d)$ for $n \geq 2d$. In particular, $\Delta_n ({\mathcal K}^2) \leq 1/ \cos {\frac \pi {2 \lfloor {n/2} \rfloor}}$ for $n \geq 4$.  
\end{abstract}

\maketitle

\section{Introduction and results}
As usual, by a \emph{convex body} of the $d$-dimensional Euclidean space $E^d$ we mean a bounded convex set with non-empty interior.
Denote by ${\mathcal K}^d$ the family of convex bodies in $E^d$ and by ${\mathcal M}^d$ the subfamily of centrally symmetric convex bodies. 
The convex hull of the union of two different parallel hyperplanes $H_1$ and $H_2$ in $E^d$ is called a \emph{strip}. 
If $H_1$ and $H_2$ are perpendicular to a direction $m$, then $S$ is said to be a \emph{strip of direction} $m$.
The distance between $H_1$ and $H_2$ is called the \emph{width} of $S$.
By the \emph{width} of a bounded set (in particular, of a convex body) $C \subset E^d$ \emph{in direction} $m$ we understand the width of the smallest strip of direction $m$ containing $C$.
The \emph{minimal width} of $C$ (in some papers called also the \emph{thickness} of $C$) is denoted by $w(C)$.
We tacitly assume that the considered polytopes are convex.

For any $C \in {\mathcal K}^d $ we set

$$\lambda_n(C) = \sup \Big\{ {\frac {w(P)} {w(C)}}; \ {\rm where} \ P \subset C \ {\rm is \ a \ polytope \ with \ at \ most}\ n \ {\rm vertices} \Big\}.$$ 

This number makes also sense in the more general situation when $P \subset C$ are any bounded sets in $E^d$. 
For instance, in~\cite{[T]} it is considered for the family of finite subsets of $E^2$ instead of ${\mathcal K}^2$. 

For every non-empty family ${\mathcal B}$ of convex bodies in $E^d$ we define

$$\Lambda_n({\mathcal B}) = \inf \{ \lambda_n (C) ; \ C \in {\mathcal B} \}.$$

Clearly, if $m \leq n$, then $\lambda_m (C) \leq \lambda_n (C)$ for every convex body $C \subset E^d$. 
Of course, $\Lambda_m({\mathcal B}) \leq \Lambda_n({\mathcal B})$ for any family ${\mathcal B}$ of convex bodies of $E^d$. 

The main aim of the paper is to give some estimates of $\Lambda_n ({\mathcal K}^d)$ and $\Lambda_n ({\mathcal M}^d)$, which means that we deal with the approximation of convex bodies by contained polytopes with at most 
$n$ vertices with respect to the minimal width (it is easy to see that the same estimates are for inscribed polytopes).
So the paper presents mostly a continuation of the research from \cite{[GL]}, where a number of estimates of this kind are given. 
For $d=2$ this task is also stated in \cite{[BMP]}, see Problem 3 on p. 452 there.

In this Section we present our results, the proofs are given in Section 4.
Only Proposition \ref{Lambda_4} is obtained in this Section as a result of a discussion on inscribing even-gons of large minimal width in a disk.

The proof of our main Theorem \ref {theorem-main} exhibits a very simple but effective method of estimating $\Lambda_n ({\mathcal K}^d)$ for every integer $n \geq 2d$. 
It is based on estimates of the smallest possible angular radius of $k$ pairs of antipodal spherical caps that cover the unit sphere of $E^d$.
It is denoted by $a_{k}^d$. 
In other words, $a_{k}^d$ is the smallest positive real for which there exist $k$ straight lines through the origin such that the angle between any straight line through the origin and at least one of these $k$ lines is at most $a_{k}^d$.
A few estimates of $a_{k}^d$ are recalled in Section 3.
Also some more are established in Lemma \ref{a_k^2}. 
In the following theorem and its application, we take 
$k= {\lfloor {\frac n 2} \rfloor}$.

\begin{thm} \label{theorem-main} 
For any integers $d \geq 2$ and $n \geq 2d$ we have

$$\Lambda_n ({\mathcal K}^d) \geq \cos a_{\lfloor {\frac n 2} \rfloor}^d.$$
\end{thm}  

The paper \cite{[GL]} gives the estimate $\Lambda_{n}({\mathcal K}^d) \geq {\frac 1 d}$ for every $n \geq d+1$. 
In Corollary \ref{Lambda2d} below for $n = 2d$ (and thus for every $n \geq 2d$) this estimate is improved.

\begin{cor} \label{Lambda2d} 
We have 
$\Lambda_{2d}({\mathcal K}^d) \geq {\frac 1 {\sqrt d}}$  for $d \geq 2$ and 
$\Lambda_{2d+2}({\mathcal K}^d)\geq \sqrt{\frac 3 {3d-2}}$ for $d \geq 3$.
\end{cor}

\begin{cor} \label{Lambda6} 
$\Lambda_6({\mathcal K}^3) \geq {\frac 1 {\sqrt 3}} \ (\approx 0.577)$, 
$\Lambda_8({\mathcal K}^3) \geq 0.654$, 
$\Lambda_{10}({\mathcal K}^3) \geq 0.695$, 
$\Lambda_{12}({\mathcal K}^3) \geq {\frac 1 {\sqrt 3}}\cot {\frac \pi 5} \ (\approx 0.794)$, 
$\Lambda_{14}({\mathcal K}^3) \geq 0.806$,   
$\Lambda_{16}({\mathcal K}^3) \geq  0.833$.   
\end{cor}

\begin{cor} \label{Lambda_n} 
For every integer $n\geq 4$ we have $\Lambda_n ({\mathcal K}^2) \geq \cos {\frac \pi {2 \cdot \lfloor {n/2} \rfloor}}$.   
\end{cor}

Clearly, this estimate for every even $n\geq 4$ says nothing other than that $\Lambda_n ({\mathcal K}^2)$ and  $\Lambda_{n+1} ({\mathcal K}^2)$ are at least $\cos {\frac \pi n}$. 
A question is how to get an estimate for $n+1$ better than $\cos {\frac \pi n}$ for an even $n$. 
In particular, an estimate for $n+1=5$ better than $\frac {\sqrt 2} 2$ for $n=4$.

In \cite{[GL]} it is conjectured that $\Lambda_3 ({\mathcal K}^2) = 6/(3+\sqrt 3 \tan 72^\circ) \approx 0.720$ and that this is attained for the regular pentagon. 
Moreover, it is shown there that $\Lambda_3 ({\mathcal K}^2) > 0.583$.
We improve this estimate in the second part of the following proposition.

\begin{pro} \label{Lambda_3} 
Every planar convex body $C$ contains a regular triangle of minimal width ${\frac 1 2}(3- \sqrt 3)\cdot w (C)$ and
$\Lambda_3 ({\mathcal K}^2) \geq {\frac 1 2}(3- \sqrt 3)$ (this number is approximately $0.634$).  
\end{pro}

Observe that the regular triangle from Proposition \ref{Lambda_3} not always may be enlarged to a regular triangle inscribed in $C$.

In connection with the first statement of Proposition \ref{Lambda_3} we recall two facts.
First, for bodies of constant width, the estimate increases up to approximately $0.73911$, as proved by Eggleston and Taylor \cite{[ET]}. 
Secondly, every planar convex body of unit minimal width contains a square of minimal width $4- 2\sqrt 3 \approx 0.536$ as shown by Eggleston \cite{[E]}. 

A natural question appears what size regular triangle (respectively, a square) of any given direction of a side can be placed in every planar convex body of unit minimal width. 
In other words, we ask how large regular triangle (respectively, square) may be ``rotated" in every such a body.  
By later Claim \ref{width-reduced}, it is sufficient to consider these questions for reduced bodies (see the definition in Section 3).
The author conjectures that the regular triangle of minimal width $0.5$ and the square of minimal width $\sqrt {2- \sqrt 3} = {\frac 1 2}(\sqrt 6 - \sqrt 2) \approx 0.518$ may be ``rotated inside" every convex body of minimal width $1$, and that in both cases the regular triangle is the worst body (i.e., that the ``rotated" triangle and square cannot be enlarged). 

Now consider finding wide polygons in a centrally symmetric convex body.
Recall that by part (2) of Theorem of \cite{[GL]} the value $\Lambda_n ({\mathcal M}^d)$ is attained for balls.  
In particular, finding $n$-gons of the largest minimal width in a disk $D \subset E^2$ 
gives the lower bound of $\Lambda_n ({\mathcal M}^2)$.
Of course, we may limit the consideration to the case when $D$ is of the unit minimal width. 

\begin{pro} \label{inscribed n-gon} 
 If $n \geq 3$ is odd, then the $n$-gon of the largest minimal width inscribed in the disk $D \subset E^2$ is only the regular $n$-gon. 
We have $\Lambda_n ({\mathcal M}^2) = \lambda_n (D) = {\frac 1 2} + {\frac 1 2}\cos {\frac 1 n}\pi$. 
\end{pro}

Proposition \ref{inscribed n-gon} is not true for even $n$.
Just immediately for $n=4$ we may be surprised that the square $S$ inscribed in $D$ is not the best approximating inscribed quadrangle.
The reason is that $w(S) = {\frac 1 2}\sqrt 2 \approx 0.707$ and for the inscribed regular triangle $T$ (still it is a degenerated quadrangle) the value $w(T) = 0.75$ is better.
What is more, we have $\lambda_4 (D) \geq {\frac 8 9} \sqrt 3 \approx 0.7698$ (and we conjecture that $\lambda_4 (D) = {\frac 8 9} \sqrt 3 $).
This estimate is realized for the deltoid $v_1v_2v_3v_4$ inscribed in the disk $D$ of the minimal width $1$ centered at $(0,0)$ (later, we always take this $D$), where $v_1 = ({\frac 1 2},0)$, $v_2= (- {\frac 1 6}, {\frac 1 3}\sqrt 2)$, $v_3 = (-{\frac 1 2},0)$ and $v_4 = (- {\frac 1 6}, -{\frac 1 3}\sqrt 2)$, see Figure 1.

{\ }

\begin{center}
\includegraphics[width=7.92cm,height=7.2cm]{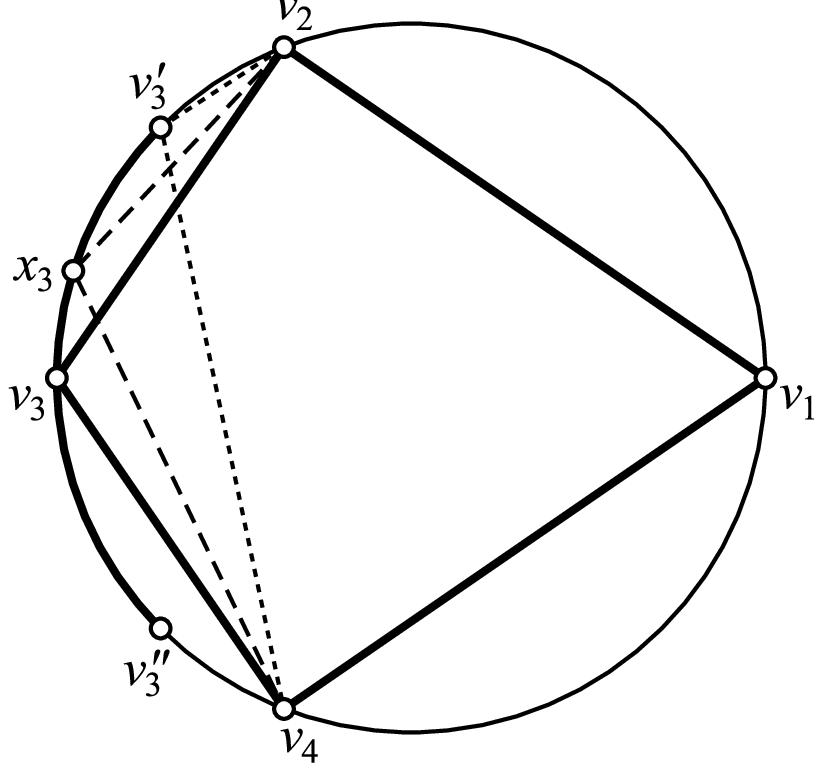}\\ %[width=8.4cm,height=8.0cm]

\vskip0.2cm
\centerline
{Fig 1. Wide quadrangles inscribed in a disk}
\end{center}

\vskip0.3cm
Also when we change the position of $v_3$ in $D$ not too much, the new quadrangle still has the minimal width ${\frac 8 9} \sqrt 3$; namely the general position $x_3$ of the vertex $v_3$ of the inscribed quadrangle $v_1v_2x_3v_4$ (in Figure 1 with two sides marked by broken lines) must be on the circle bounding $D$ so that the distances from $v_1$ to the straight lines carrying $v_2x_3$ and $x_3v_4$ remain at least ${\frac 8 9} \sqrt 3$, i.e., that each of the angles $\angle v_1v_2x_3$ and $\angle v_1v_4x_3$ must be at least $\angle v_4v_1v_2$.
The extreme positions $v'_3$ and $v''_3$ of $x_3$ give two quadrangles $v_1v_2v_3'v_4$ (in Figure 1 with two sides marked by pointed lines) and $v_1v_2v_3''v_4$, each of which is an inscribed trapezium with three sides of length $\sqrt {2/3} \approx 0.8164$.

A widest found hexagon in the disk $D$ has vertices $v_i = ({\frac 1 2}\cos \alpha_i, {\frac 1 2}\sin \alpha_i)$ for $i= 1, \dots ,6$, where 
$\alpha_1 = {\rm arc cos} {\frac {\sqrt {145} -5} {20}} \approx (69.385)^\circ$, 
$\alpha_2 = 2\alpha_1 \approx (139.77)^\circ$, 
$\alpha_3 = 180^\circ$, 
$\alpha_4 = 360^\circ - \alpha_2 \approx (221.229)^\circ$, 
$\alpha_5 = 360^\circ - \alpha_1 \approx (290.615)^\circ$ 
and $\alpha_6 = 360^\circ$.
Applying the fact that the minimal width of any polygon is realized for the direction perpendicular to a side of it and using the formula for the distance between a point and a  straight line, by an easy but tedious evaluation we find that the minimal width of the hexagagon under study is $({\frac 1 2} + {\frac 1 2} \cos \alpha_1 -\cos^2\alpha_1)\sqrt {2 +2 \cos \alpha_1} \approx 0.90786$. 
The author expects that it gives the value of $\lambda_6 (D) = \Lambda_6 ({\mathcal M}^2)$.  
By the way, the minimal width of this hexagon does not change when we permit $\alpha_3$ to vary between 
$360^\circ - 3 \alpha_1 \approx (151.845)^\circ$ and $3 \alpha_1  \approx (208.155)^\circ$.  
For the first extreme position of $\alpha_3$ the side $v_2v_3$ is parallel to $v_5v_6$, and for the second the side $v_3v_4$ is parallel to $v_6v_1$.

Finding even-gons of large minimal width with more vertices becomes more complicated. 
The author tried to find such large octagons inscribed in $D$ with a computer approach. 
This computation shows that
$\lambda_8 (D) = \Lambda_8 ({\mathcal M}^2)$ is at least $0.95143$, namely by taking the inscribed octagon with vertices $v_i = ({\frac 1 2}\cos \alpha_i, {\frac 1 2}\sin \alpha_i)$ for $i=1, \dots , 8$, where $\alpha_1, \dots , \alpha_8$ are $(50.432)^\circ$, $(100.864)^\circ$, $(151.296)^\circ$, $180^\circ$, $(208.704)^\circ$, $(259,136)^\circ$, $(309.568)^\circ$, $360^\circ$.

Recapitulating, we obtain the following proposition.

\begin {pro} \label{Lambda_4}  
We have
$\Lambda_4({\mathcal M}^2) = \lambda_4(D) \geq {\frac 8 9} \sqrt 3 \approx 0.7698$, 
$\Lambda_6 ({\mathcal M}^2) = \lambda_6 (D) \geq 0.90786$, 
$\Lambda_8 ({\mathcal M}^2) = \lambda_8 (D) \geq 0.95143$.
\end{pro}

An open question remains about analogical estimates of $\Lambda_n({\mathcal M}^d)$ better than those for $\Lambda_n({\mathcal K}^d)$ obtained in Theorem \ref{theorem-main} and Corollaries \ref{Lambda2d} and \ref{Lambda6}. 
In particular, one can look for for a better estimate for ${\mathcal M}^d$ than the one in the first statement of Corollary \ref{Lambda2d} (at least for $d=3$) in analogy to the $2$-dimensional example of the deltoid (see the paragraph just after Proposition \ref{inscribed n-gon}) which implied $\lambda_4(D) \geq {\frac 8 9}\sqrt 3$. 

Finally, we consider the dual problem of estimating from above the diameter of the polytopes containing a convex body of a given diameter.
For $d=2$ it is mentioned in \cite{[BMP]} on p. 452.
For this aim let us introduce the number

$$\delta_n (C) = \inf \Big\{ {\frac {{\rm diam}(P)} {{\rm diam}(C)}}; \ {\rm where} \ P \supset C \ {\rm is \ a \ polytope \ with \ at \ most} \ n \ {\rm facets} \Big\}.$$ 

\noindent
Here $C \in {\mathcal K}^d$, and the symbol diam stands for the diameter.   
For any family ${\mathcal B} \subset {\mathcal K}^d$ we put

$$\Delta_n({\mathcal B}) = \sup \{ \delta_n (C) ; \ C \in {\mathcal B} \}.$$

\begin{thm} \label{Delta_n} 
For any integers $d \geq 2$ and $n \geq 2d$ we have

$$\Delta_n ({\mathcal K}^d) \leq {\frac 1 {\cos a_{\lfloor {\frac n 2}\rfloor}^d}}.$$
\end{thm}  

\begin{cor} \label{Delta_6}
We have $\Delta_6 ({\mathcal K}^3) \leq 1.773$, $\Delta_8 ({\mathcal K}^3) \leq 1.529$, $\Delta_{10} ({\mathcal K}^3) \leq 1.438$, $\Delta_{12} ({\mathcal K}^3) \leq 1.257$, $\Delta_{14} ({\mathcal K}^3) \leq 1.239$ and $\Delta_{16} ({\mathcal K}^3) \leq 1.199$. 
Moreover, $\Delta_{2d} ({\mathcal K}^d) \leq \sqrt d$ for $d\geq 2$ and $\Delta_{2d+2} ({\mathcal K}^d) \leq \sqrt {d - {\frac 2 3}}$ for $d\geq 3$.
\end{cor}

\begin{cor} \label{polygon-containing}
For every convex body $C\subset E^2$ and every integer $n\geq 4$ there exists a polygon $P$ with at most $2{\lfloor {\frac n 2}\rfloor}$ sides (so with at most $2{\lfloor {\frac n 2}\rfloor}$ vertices) containing $C$ whose diameter is at most $1/\cos {\frac \pi {2 \lfloor {\frac n 2}\rfloor}}$.
\end{cor}

If we take only even $n$, this corollary reads: {\it for every convex body $C\subset E^2$ and every even $n\geq 4$ there exists a polygon $P$ with $n$ sides (so with at most $n$ vertices) containing $C$ whose diameter is at most $1/\cos {\frac \pi n}$.}
It is easy to check that the diameter of the circumscribed regular $n$-gon about $B$, for $n$ even, is $1/\cos {\frac \pi n}$.
So for the regular even-gon circumscribing $B$ the value from Corollary \ref{polygon-containing} is attained. 
We conjecture that this estimate $1/\cos {\frac \pi n}$ cannot be improved for the disk $B$ in the part of~$C$. 
Let us add that for $n=3$ this is true by \cite{[G]}.
The author believes that also for every odd $n$ the best (i.e., of the smallest diameter) convex $n$-gon circumscribing the disk $B \subset E^2$ is the regular (for which the ratio of the diameters equals $\sin {\frac \pi n} / \cos {\frac \pi {2n}}$).

%--------------------- SECTION 2

\section{Auxiliary results on diametral chords and reduced bodies}

The longest chords of a convex body $C \subset E^d$ in a given direction are called \emph{ diametral chords of $C$} in that direction. 
By compactness arguments there is at least one diametral chord of $C$ in arbitrary direction.
By Part 33 of \cite{[BF]} or by Theorem 12.18 of \cite{[V]} for any convex body $C \subset E^d$, the minimum of lengths of diametral chords of $C$ equals $w(C)$, which implies the following claim.

\begin{cla} \label{diametral-chord} 
Every diametral chord of any convex body $C$ has length at least $w(C)$. 
\end{cla}

\begin{lem} \label{two-chords} %Lemma 3.1 
If two diametral chords of a convex body $C \subset E^2$ are not parallel, then they intersect.
\end{lem}

\begin{proof}
Assume that some two non-parallel diametral chords $ab$ and $cd$ do not intersect, where $a, b, d, c$ are in this order on the boundary of $C$. 
Since $C$ is convex, $abdc$ is a convex quadrangle.
If $\angle acd + \angle cdb > \pi$, then  $cd$ cannot be a longest chord of $C$ in its direction. 
If $\angle cab + \angle abd > \pi$, then $ab$ cannot be a longest chord of $C$ in its direction.
So $abdc$ must be a non-degenerated rectangle.
Hence $ab$ and $cd$ are parallel diametral chords, which contradicts the assumption of our lemma. 
\end{proof}

\begin{lem} \label{continuously}  
Assume that for a convex body $C \subset E^2$ and every direction $\ell$ there is only one diametral chord in direction $\ell$. 
Then the position of the diametral chord of $C$ changes continuously as $\ell$ varies.
\end{lem}

\begin{proof} 
As explained in \cite{[L2]} (see the paragraph just after Lemma 1 there) the length of the diametral chord is a continuous function of its direction. 
So if $\ell_0$ is the limit of directions $\ell$, then taking also into account the assumption of our lemma, we conclude that the limit of the diametral chords of $C$ in these directions $\ell$ is the diametral chord in direction $\ell_0$, which confirms the stated continuity. 
\end{proof}

Recall that a convex body $R$ of Euclidean $d$-space $E^d$ is called \emph{reduced} if for every convex body $Z \subset R$ different from $R$ we have $w (Z) < w (R)$.
Denote by ${\mathcal R}^d$ the family of reduced bodies in $E^d$. 
For a survey of results on reduced bodies see \cite{[LM]}. 
Later we apply reduced bodies in the proof of Proposition \ref{Lambda_3}. 
This application is based on the following claim from \cite{[LM]}.

\begin{cla} \label{width-reduced}  
Every convex body contains a reduced convex body of the same minimal width. 
\end{cla} 

This claim guarantees that $\Lambda_n({\mathcal R}^d) = \Lambda_n({\mathcal K}^d)$, which permits to work with this narrower class of bodies. 
What is more, every reduced planar body $R$ has diameter at most $\sqrt 2 \cdot w (R)$ (see \cite{[L1]}), as opposed to the diameter of an arbitrary convex body which may be arbitrarily large.
An additional advantage of reduced bodies is the property shown in the following lemma, which is later applied in the proof of Proposition~\ref{Lambda_3}.

\begin{lem} \label{direction}   
Let $R \subset E^2$ be a reduced body. 
In every direction there is exactly one diametral chord of $R$. 
\end{lem}

\begin{proof}
Assume that there are two different diametral chords of $R$ in a direction. 
Then the boundary of $R$ contains the two parallel segments connecting the end-points of these chords. 
But the boundary of $R$ does not contain any two parallel segments, as it follows from Theorem 4, part (b) of \cite{[L1]}.
A contradiction.
\end{proof}

%--------------------- SECTION 3

\section{Some estimates of $a^d_k$}

In \cite{[FTG]}, especially see p. 2286 there, a few values of $a_k^3$ (denoted there by $r_{2N}^*$, where $k=N$) are estimated or remembered, in particular from \cite{[F1]}.  
Recall after \cite{[FTG]} that for instance $a_3^3 \leq (54.7356)^\circ$, $a_4^3 \leq (49.1066)^\circ$, $a_{5}^3 \leq (45.9243)^\circ$, $a_{6}^3 \leq (37.3774)^\circ$, $a_{7}^3 \leq (36.2060)^\circ$ and $a_{8}^3 \leq (33.5473)^\circ$.

In the following lemma we present more estimates of the numbers $a^d_k$.

\begin{lem} \label{a_k^2} 
We have $a^2_m = {\frac \pi {2m}}$ for every integer $m \geq 2$, $a_{d}^d \leq {\rm arc cos}{\frac 1 {\sqrt d}}$ for $d\geq 2$, and $a_{d+1}^d \leq {\rm arc cos} \sqrt{\frac 3 {3d-2}}$ for $d\geq 3$. 
\end{lem}

\begin{proof} The first statement is obvious.

Let us show the second one. 
Take $d$ orthogonal straight lines $L_1, \dots , L_d$ through the origin (say, the axes of a Cartesian coordinate system). 
Take an arbitrary straight line $L$ through the origin and let $[u_1, \dots , u_d]$ be one of the two unit vectors parallel to $L$.  
Without loss of generality, we may assume that for instance the first axis of the coordinate system is the line from amongst our lines $L_1, \dots , L_d$ for which the angle with $L$ is the smallest.
So the angle between $[1, 0, \dots , 0]$ and $[u_1, \dots , u_d]$ (or $[-u_1, \dots , -u_d]$) is the smallest. 
This and $u_1^2 + \dots + u_d^2 =1$ imply that $u_1^2$ is a largest number from amongst $u_1^2, \dots , u_d^2$.
Hence $u_1^2 \geq {\frac 1 d}$.
Moreover, since the angle between $[u_1, \dots , u_d]$ and $[1, 0, \dots , 0]$ is at most the angle between $[u_1, \dots , u_d]$ and $[-1, 0, \dots , 0]$, we have $u_1 \geq 0$ and thus $u_1 \geq {\frac 1 {\sqrt d}}$.  
So $[u_1, \dots , u_d] \circ [1, 0, \dots , 0]  \geq {\frac 1 {\sqrt d}}$ which gives $a_d^d \leq {\rm arc cos}{\frac 1 {\sqrt d}}$, i.e., the second statement.

In order to show the third statement replace the line $L_2$ by the following two lines through the origin: $L'_2$ containing $({\frac 1 2}, {\frac {\sqrt 3} 2}, 0, \dots , 0)$, and $L''_2$ containing $(-{\frac 1 2}, {\frac{\sqrt 3} 2}, 0, \dots , 0)$. 
Take an arbitrary straight line $L$ through the origin. 
We will show that the angle between $L$ and at least one of the lines $L_1, L'_2, L''_2, L_3, \dots , L_d$ is at most ${\rm arc cos} \sqrt{\frac 3 {3d-2}}$.
If the angle between $L$ and at least one of the lines $L_3, \dots , L_d$ fulfills this condition, the situation is clear.

Since now assume the opposite, namely that the angle between $L$ and every of the lines $L_3, \dots , L_d$ is over ${\rm arc cos} \sqrt{\frac 3 {3d-2}}$.
So for every $i \in \{ 3, \dots , d \}$ we have $|u_i| < \sqrt{\frac 3 {3d-2}}$.
Hence $u_1^2 + u_2^2 > 1 - (d-2){\frac 3 {3d-2}} = {\frac 4 {3d-2}}$.

The projection of the point $(u_1, \dots ,u_d)$ on the two-dimensional plane $Ox_1x_2$ is on a circle $x_1^2 + x_2^2 = b^2$, where $\sqrt{\frac 4 {3d-2}} < b \leq 1$. 
Without loss of generality, we may assume that it is on the shorter piece of the circle ``between" $(b, 0)$ and $({\frac {\sqrt 3} 2}b, {\frac 1 2}b)$.
Then $[u_1, \dots , u_d] \circ [1, 0, \dots , 0] = u_1 \geq {\frac {\sqrt 3} 2}b >  {\frac {\sqrt 3} 2} \sqrt{\frac 4 {3d-2}}   = \sqrt{\frac 3 {3d-2}}$.
Consequently, the angle between $L$ and $L_1$ is at most ${\rm arc cos} \sqrt{\frac 3 {3d-2}}$, which ends the proof.
\end{proof}

\vskip 2cm

%--------------------------SECTION 4

\section{Proofs of theorems, propositions and corollaries}

\hskip 0.5cm{\it Proof of Theorem} \ref{theorem-main}.

\medskip
Let $C \subset E^d$ be a convex body.
By the definition of $a^d_k$ for $k = {\lfloor {\frac n 2} \rfloor}$ there exist straight lines $L_1, \dots , L_{\lfloor {\frac n 2} \rfloor}$ through the center of the unit sphere, each of which passes through the centers of a pair from amongst of the $\lfloor {\frac n 2} \rfloor$ pairs of antipodal spherical caps that cover the sphere.
Take a diametral chord $D_i$ of $C$ parallel to $L_i$ for $i=1, \dots, \lfloor {\frac n 2}\rfloor$.
By Claim \ref{diametral-chord} 
the length of every of these diametral chords is at least $w(C)$.
Construct the convex hull $H(D_1, \dots , D_{\lfloor {\frac n 2}\rfloor})$ of the union these $\lfloor {\frac n 2} \rfloor$ diametral chords. 

Of course, $H(D_1, \dots , D_{\lfloor {\frac n 2}\rfloor})$ is a polytope with  
$2 \lfloor {\frac n 2} \rfloor$ or fewer vertices inscribed in $C$ (just observe that possibly some end-points of our diametral chords coincide).

Take an arbitrary straight line $L$ through the origin and the narrowest strip $S(L)$ orthogonal to $L$ which contains $H(D_1, \dots , D_{\lfloor {\frac n 2} \rfloor})$. 
Of course, $S(L)$ contains the diametral chords $D_1, \dots , D_{\lfloor {\frac n 2} \rfloor}$.
By the description of the number $a^d_k$ for $k = {\lfloor {\frac n 2} \rfloor}$, 
there exists a straight line $L_i$, where $i \in \{ 1, \dots, {\lfloor {\frac n 2} \rfloor} \}$, such that the angle between $L_i$ and $L$ is at most  $a_{\lfloor {\frac n 2} \rfloor}^d$.
What is more, the strip $S(L)$ contains the diametral chord $D_i$, and the length of $D_i$ is at least $w(C)$.
Consequently, the width of the strip $S(L)$ is at least $w (C) \cdot \cos  a_{\lfloor {\frac n 2} \rfloor}^d$.
Since this holds true for every straight line $L$ through the origin, we conclude that $\lambda_n (C) \geq \cos a_{\lfloor {\frac n 2} \rfloor}^d$. 
By the arbitrariness of the convex body $C$ we draw the conclusion of our theorem.

\medskip
\bigskip
For every $C$ there is a wide class of the polytopes $H(D_1, \dots , D_{\lfloor {\frac n 2}\rfloor})$ constructed in this proof.
Still we may take miscellaneous families of lines $L_1, \dots , L_{\lfloor {\frac n 2} \rfloor}$ (in particular by rotating a fixed good family of them).
Moreover, diametral chords in a direction may be not are unique.

\bigskip
{\it Corollary \ref{Lambda2d}} results from Theorem \ref{theorem-main} and Lemma \ref{a_k^2}.  

\bigskip
{\it Corollary \ref{Lambda6}} follows from Theorem \ref{theorem-main} and the estimates listed at the beginning of Section 3 for even $n$ between $6$ and $16$. 

\bigskip
{\it Corollary \ref{Lambda_n}} results from Theorem \ref{theorem-main} for $d=2$ and $n\geq 4$, by considering $\lfloor {n/2} \rfloor$ diametral chords of $C \in {\mathcal K}^2$ with successive angles ${\frac 1 {\lfloor {n/2} \rfloor}}\pi$, and by applying the first statement of Lemma \ref {a_k^2} for $n = {\lfloor {n/2} \rfloor}$.

\bigskip
{\it Proof of Proposition} \ref{Lambda_3}.

\medskip
Clearly, it is sufficient to show the first statement of the proposition.

By Claim \ref{width-reduced} there is a reduced body $R \subset C$ with $w(R) = w(C)$. 

For every direction $\delta$ take the diametral chord $ac$ of $R$ in this direction. 
It is unique, which results by Lemma \ref{direction}.  
By Claim \ref{diametral-chord} its length is always at least $w(C)$. 
Also in the perpendicular direction take the unique diametral chord of~$R$.
By Lemma \ref{two-chords} these diametral chords intersect.
Denote by $b$ the point of intersection of them. 
We will show that there exists a direction $\delta_0$ for which $|a_0b_0| = |b_0c_0|$. 

Here $a_0$, $b_0$ and $c_0$ are the specific positions of $a, b, c$ for $\delta_0$.

Take any particular direction $\delta'$ in the role of $\delta$. 
Denote by $a', b', c'$ the positions of $a, b, c$. 
If $|a'b'| = |b'c'|$, we are done. 
Assume that $|a'b'| \not = |b'c'|$.
Let for instance $|a'b'| < |b'c'|$.
We rotate $\delta$ starting from $\delta'$ until we obtain 
again $\delta'$ after rotating by $\pi$. 
All the time the positions of the end-points of $ab$ change continuously (see Lemmas \ref{continuously} and \ref{direction}).
At the end of the rotation process we get $a''c''$, which is nothing else but $c'a'$.
Clearly $|a''b''| > |b''c''|$.
From the continuity we conclude that there exists a direction $\delta_0$ for which $|a_0b_0| = |b_0c_0|$.
\break

\eject

\begin{center}
\includegraphics[width=8.19cm,height=7.92cm]{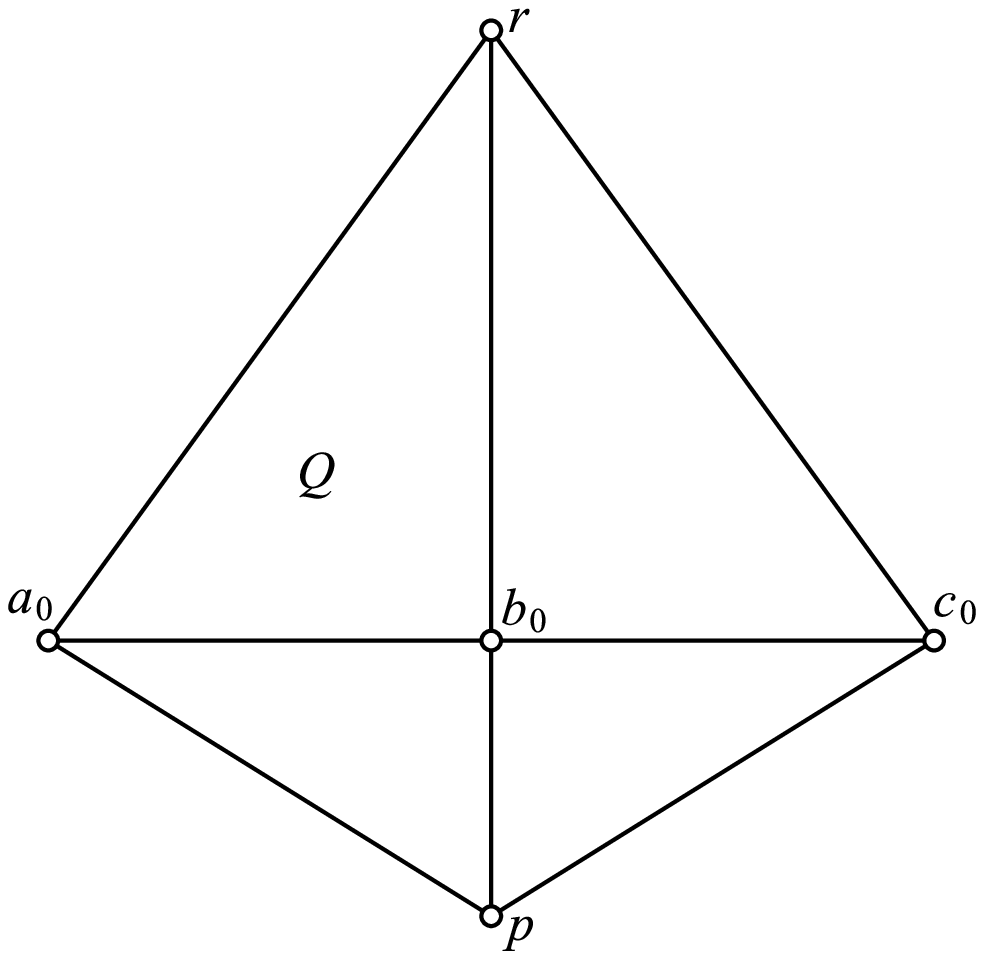}\\ %[width=9.1cm,height=8.8cm]

\vskip0.2cm
\centerline
{Fig 2. Illustration to the proof of Proposition \ref{Lambda_3}}
\end{center}

\vskip0.3cm

Consider the quadrangle $Q$ which is the convex hull of the diametral chord $a_0c_0$ and the perpendicular diametral chord $pr$ of $R$.
This diametral chord is unique by Lemma \ref{direction}. 
By Claim \ref{diametral-chord} the lengths of both these chords are at least $w(R)$. 
Of course, $pr$ intersects $a_0c_0$ at $b_0$. 
We will show that $Q$ contains a regular triangle of minimal width ${\frac 1 2}(3- \sqrt 3)\cdot w (R)$.
Clearly, later it is sufficient to consider the worst case when $Q$ is as small as possible, that is $|pr| = |a_0c_0| = w(R)$.
Let for instance $|pb_0| \leq |b_0r|$ (see Figure 2). 

For further evaluation put our points in a rectangular coordinate system so that $a_0 = (-{\frac 1 2}\cdot w(R), 0)$, $b_0 =(0,0)$, $c_0 =({\frac 1 2}\cdot w(R),0)$, $p = (0, -k\cdot w(R))$ and $r = (0, (1-k)\cdot w(R))$, where $k$ is a number from the interval $[0,{\frac 1 2}]$.

Case 1, when $k \in [0, 1 - {\frac {\sqrt 3} 2}]$.
Observe that $Q$ contains the regular triangle of side $w(R)$.
Namely the triangle with vertices $a_0(-{\frac 1 2}\cdot w(R), 0)$, $c_0({\frac 1 2}\cdot w(R), 0)$ and $v = (0, {\frac {\sqrt 3} 2}\cdot w(R))$.
The minimal width ${\frac {\sqrt 3} 2}\cdot w(R)$ of this triangle is larger than ${\frac 1 2}(3- \sqrt 3)\cdot w(R)$.

Case 2, when $k \in (1 - {\frac {\sqrt 3} 2}, {\frac 1 2}]$.
Consider the regular triangle $T$ with vertex $r$ and the opposite side parallel to $a_0c_0$ with end-points in the segments $a_0p$
and $pc_0$. 
Of course, $T \subset Q \subset R$. 
Let us show that $T$ has the smallest size if 
the midpoint of $pr$ is $b_0$ and that then its minimal width is ${\frac 1 2}(3- \sqrt 3)\cdot w (R)$.

Consider the line containing $pc_0$, i.e., the line $y=2k(x-{\frac 1 2}\cdot w(R))$.
Take a point $q=(x, 2k(x - {\frac 1 2}\cdot w(R))$ in $pc_0$.
An evaluation shows that $qrq'$, where $q' = (x, -2k(x - {\frac 1 2}\cdot w(R))$, is a regular triangle for $x_0 = {\frac {w(R)} {\sqrt 3 + 2k}}$. 
Of course, the length of the sides of this triangle is $f(k)= {\frac {2w(R)} {\sqrt 3 + 2k}}$. 
This is a decreasing function of $k$. 
Hence it attains the smallest value in $(1 - {\frac {\sqrt 3} 2}, {\frac 1 2}]$ for $k= {\frac 1 2}$, i.e., when the midpoint of $pr$ is at $b_0$. 
Its value is $f({\frac 1 2}) =  {\frac {2w(R)} {\sqrt 3 +1}}$.
So the minimal width of this triangle is ${\frac {\sqrt 3} 2} \cdot {\frac {2w(R)} {\sqrt 3 +1}} = {\frac 1 2}(3-\sqrt 3)\cdot w(R)  = {\frac 1 2}(3-\sqrt 3)\cdot w(C)$.

\bigskip
{\it Proof of Proposition} \ref {inscribed n-gon}.

\medskip
Take any non-regular $n$-gon $Q_n$ inscribed in $D$. 
Clearly, one of its sides must be longer than the sides of the regular $n$-gon $R_n$ inscribed in $D$.
Observe that the width of $Q_n$ perpendicular to this side is smaller than $w (R_n)$.
The reason is that this width is at most the width of the strip containing $Q_n$ between the straight line containing this side and the corresponding parallel supporting line of the disk.
We omit an easy calculation which confirms the formula from the second statement.

\bigskip 
{\it Proof of Theorem} \ref {Delta_n}.

\medskip
It is sufficient to show that for arbitrary $C \in {\mathcal K}^d$ we have $\delta_n (C) \leq  {1/\cos a_{\lfloor {\frac n 2}\rfloor}^d}$.
Of course, we may limit the consideration to the case when the diameter of $C$ is one, so let ${\rm diam} (C) =1$ since now. 
It is well known that there exists a body $W$ of unit constant width containing~$C$ (e.g., see \cite{[BF]}, section 64). 

By the definition of the number $a^d_k$, taking $k= {\lfloor {\frac n 2}\rfloor}$ we see that there exists a family $\mathcal L$ of ${\lfloor {\frac n 2}\rfloor}$ straight lines $L_1, \dots , L_{\lfloor {\frac n 2}\rfloor}$ through the origin $o$ such that the angular distance between an arbitrary line through the origin and at least one of these lines is at most $a^d_{\lfloor {\frac n 2}\rfloor}$.
For every $i \in \{1, \dots , a^d_{\lfloor {\frac n 2}\rfloor}\}$, take the strip $S_i(W)$ of width $1$ orthogonal to $L_i$ which contains~$W$.
The intersection $I(W)$ of these strips is a polytope with at most $2{\lfloor {\frac n 2}\rfloor}$ facets. 
Clearly, $C \subset I(W)$.

{\ }

\begin{center}
\includegraphics[width=10.53cm,height=7.2cm]{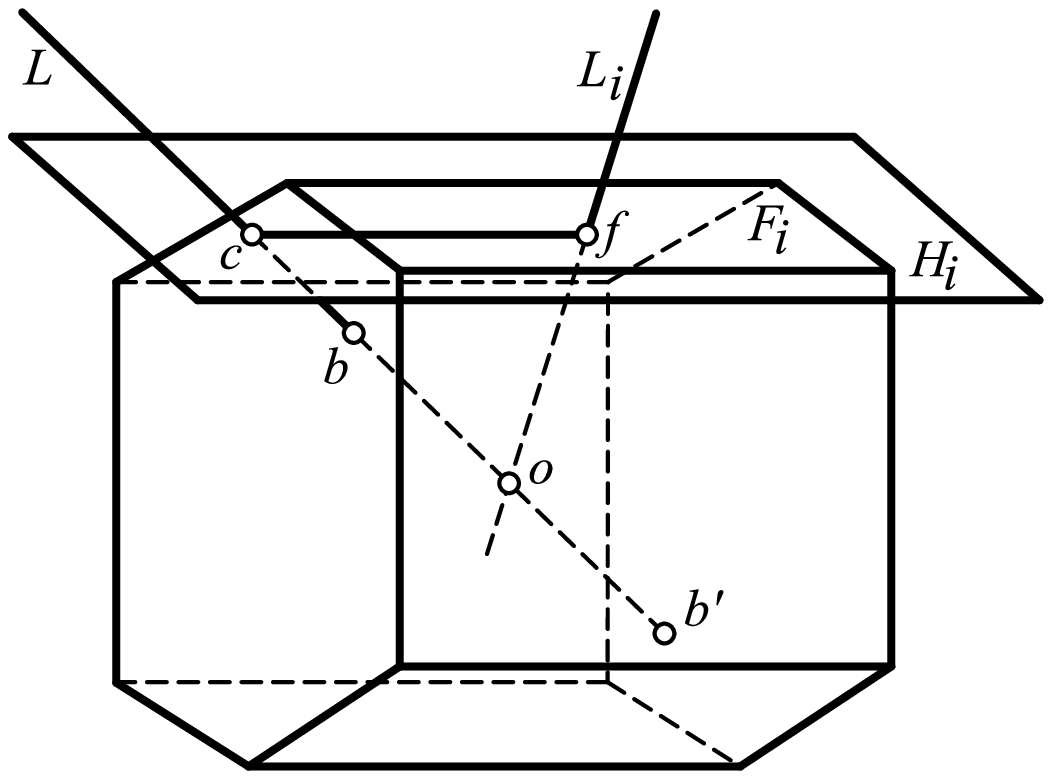}\\ %[width=11.7cm,height=8.0cm]

\vskip0.2cm
\centerline
{Fig 3. Illustration to the proof of Theorem \ref{Delta_n}}
\end{center}

\vskip0.3cm

Denote by $K^*$ the image of a convex body $K$ under the central symmetrization. 
Since $W$ is a body of constant width, $W^*$ is the ball of unit width.

Clearly, $I(W)^*$ is a centrally symmetric polytope circumscribed about the ball $W^*$.
It has ${\lfloor {\frac n 2}\rfloor}$ pairs of symmetric facets orthogonal to the lines from $\mathcal L$. 
It is well known that every body and its centrally symmetrized body have equal diameters (e.g., see \cite{[BF]}, section 42).
Thus ${\rm diam}(I(W)^*) = {\rm diam}(I(W))$. 

Let $b$ be a boundary point of $I(W)^*$.
Let $L$ be the line through $o$ and $b$.
Take a line $L_i \in \mathcal L$ whose angle with $L$ is at most $a_{\lfloor {\frac n 2}\rfloor}^d$ (see Figure 3).
By the definition of $I(W)^*$ and since it is centrally symmetric we see that there exists a pair of facets  of $I(W)^*$ orthogonal to $L_i$.
Let $F_i$ be one of these two facets.
Observe that $L_i$ intersects $F_i$. 
The reason is that $F_i$ supports the ball $W^*$ at exactly one point $f$ and $L_i$ passes through $f$. 

Denote by $H_i$ the hyperplane carrying $F_i$.
Since $H_i$ is a supporting hyperplane of $I(W)^*$, we see that $b$ is on the same side of the hyperplane $H_i$ as $o$.
In particular, $b$ may be on $H_i$. 

Recall that $\mathcal L$ contains at least $d$ lines.
Since for $m \leq n$ we have $a^d_m \geq a^d_n$, and since by Lemma \ref{a_k^2} we have $a^d_d < {\frac \pi 2}$, we conclude that the angle between $L$ and $L_i$ is below $\frac \pi 2$. 
Thus there exists a point $c \in H_i$ such that $b \in oc$.

Of course, $|ob| \leq |oc|$.
From the right triangle $ofc$ and from $|of| = {\frac 1 2}$ (this follows from the fact that $f$ is a boundary point of the ball $W^*$) we see that $|oc| \leq {1/2\cos a_{\lfloor {\frac n 2}\rfloor}^d}$.
From $|ob| \leq |oc|$ we get $|ob| \leq {1/2\cos a_{\lfloor {\frac n 2}\rfloor}^d}$.
Thus $|bb'| \leq {1/\cos a_{\lfloor {\frac n 2}\rfloor}^d}$ for the symmetric boundary point $b'$ of $I(W)^*$.
Since $I(W)^*$ is centrally symmetric, from the arbitrariness of its boundary point $b$ we conclude that ${\rm diam}(I(W)^*) \leq {1/\cos a_{\lfloor {\frac n 2}\rfloor}^d}$.
Hence ${\rm diam}(I(W)) \leq {1/\cos a_{\lfloor {\frac n 2}\rfloor}^d}$.
Thus $\delta_n (I(W)) \leq {1/\cos a_{\lfloor {\frac n 2}\rfloor}^d}$.
Since $C \subset W \subset I(W)$ and ${\rm diam} (C) =1$, this implies $\delta_n (C) \leq {1/\cos a_{\lfloor {\frac n 2}\rfloor}^d}$, which ends the proof.

\bigskip
{\it Corollary \ref {Delta_6}} follows by Theorem \ref{Delta_n} taking into account the estimates presented in Section 2 of this paper.

\bigskip
{\it Corollary \ref{polygon-containing}} is a result of Theorem \ref{Delta_n} and the first statement of Lemma \ref{a_k^2}.

\end{document}